\documentclass{amsart}

\usepackage{bbm}
\usepackage{hyperref}

\newcommand*{\mailto}[1]{\href{mailto:#1}{\nolinkurl{#1}}}
\newcommand{\arxiv}[1]{\href{http://arxiv.org/abs/#1}{arXiv:#1}}

\newtheorem{theorem}{Theorem}[section]

\newtheorem{lemma}[theorem]{Lemma}
\newtheorem{example}[theorem]{Example}

\newtheorem{corollary}[theorem]{Corollary}
\newtheorem{hypothesis}[theorem]{Hypothesis}
\newtheorem{remark}[theorem]{Remark}

\newcommand{\R}{{\mathbb R}}
\newcommand{\N}{{\mathbb N}}

\newcommand{\C}{{\mathbb C}}

\newcommand{\be}{\begin{equation}}
\newcommand{\ee}{\end{equation}}

\newcommand{\cH}{\mathcal{H}}
\newcommand{\cD}{\mathcal{D}}
\newcommand{\cF}{\mathcal{F}}
\newcommand{\cG}{\mathcal{G}}
\newcommand{\cA}{\mathcal{A}}
\newcommand{\cL}{\mathcal{L}}
\newcommand{\cJ}{\mathcal{J}}

\newcommand{\gH}{\mathfrak{H}}

\newcommand{\gt}{\mathfrak{t}}

\newcommand{\I}{\mathrm{i}}

\newcommand{\im}{\mathrm{Im\,}}

\newcommand{\dom}{\mathrm{dom}}
\newcommand{\loc}{\mathrm{loc}}
\newcommand{\ran}{\mathrm{ran}}

\numberwithin{equation}{section}

\begin{document}

\title[$J$-positive block operator matrices]{A note on $J$-positive block operator matrices}

\author[A.\ Kostenko]{Aleksey Kostenko}
\address{Faculty of Mathematics\\ University of Vienna\\
Oskar--Morgenstern--Platz 1\\ 1090 Wien\\ Austria}
%\email{\mailto{duzer80@gmail.com}; \mailto{Oleksiy.Kostenko@univie.ac.at}}

%\thanks{{To appear in}
\thanks{{\it Research supported by the Austrian Science Fund (FWF) under Grant No.\  P26060.}}


\begin{abstract}
We study basic spectral properties of $\cJ$-self-adjoint $2\times 2$ block operator matrices. Using the linear resolvent growth condition, we obtain simple necessary conditions for the regularity of the critical point $\infty$. In particular, we present simple examples of operators having the singular critical point $\infty$. Also, we apply our results to the linearized operator arising in the study of soliton type solutions to the nonlinear relativistic Ginzburg--Landau equation. 
\end{abstract}

\maketitle

\section{Introduction}\label{sec:intro}

Let $\cH$ be a complex separable Hilbert space. Consider the following operators defined in $\gH=\cH\times\cH$ by the block operator matrices
\be\label{eq:cA}
\cL=\begin{pmatrix}
\ \I C & \I B \\
\ -\I A & \ -\I C^*
\end{pmatrix},\qquad \cA=\begin{pmatrix}
A & C^* \\
C & B
\end{pmatrix}.
\ee
Note that $\cL=\cJ\cA$, where % Now let us introduce the following fundamental symmetry 
\be\label{eq:fs}
\cJ=\cJ^\ast=\cJ^{-1}=\begin{pmatrix}
0 & \I I\\
-\I I & 0
\end{pmatrix}
\ee
is a fundamental symmetry on $\cH\times\cH$. Also, $I$ stands for the identity operator on $\cH$.
%then the operator $\cA$ is $J$-self-adjoint in $\cH\times\cH$.
The operators $A$, $B$ and $C$ are not assumed to be bounded. In order to define the operators $\cA$ and $\cL$ correctly we shall assume the following.

\begin{hypothesis}\label{hyp:01}\
\begin{itemize}
\item[(i)] %$A=A^\ast>0$ 
$A$ is closed with $0\in\rho(A)$ and $\kappa_-(A)<\infty$,
\item[(ii)] $C$ is closed, $\dom(A)\subset\dom(C)$ and $\dom(A)\subset\dom(C^*)$,
\item[(iii)] $B=B^\ast$,
\item[(iv)] $\dom(S_0):=\dom(C^*)\cap\dom(B)$ is dense in $\cH$ and the operator 
\be\label{eq:S}
S_0:=B- CA^{-1}C^\ast
\ee
 is essentially self-adjoint on $\dom(S_0)$ with $\kappa_-(\overline{S_0})<\infty$.
\end{itemize}
\end{hypothesis}

Here $\kappa_-(T)=\dim\ran \chi_{(-\infty,0)}(T)$. Note that $\kappa_-(T)$ is the number of negative eigenvalues of $T$ if  $\kappa_-(T)<\infty$. 

Under the assumptions of Hypothesis \ref{hyp:01}, the operator $\cA_0$ %and $\cL_0=\cJ\cA_0$ 
defined on $\dom(\cA_0)=\dom(A)\times\dom(S_0)$ is essentially self-adjoint (see Theorem \ref{th:shkal}). If additionally the operators $A$ and $S_0$ are positive, then so is the operator $\cA_0$. Moreover, the operator $\cL_0$ defined by $\cL_0=\cJ\cA_0$ on $\dom(\cL_0)=\dom(\cA_0)$ is closable and essentially $\cJ$-self-adjoint. 
%Moreover, it is definitizable (in the sense of \cite{lan84}) if $\rho(\overline{\cL})\neq \emptyset$. 
 Let us denote by $\cA$ and $\cL$ the closures of $\cA_0$ and $\cL_0$, respectively.

Operators $\cA$ and $\cL$ arise in various areas of mathematical physics and hydrodynamics. In particular, the spectral properties of $\cA$ has been studied in \cite{shk95}, \cite[Chapter II]{tre} (see also references therein). Note also that our choice of the fundamental symmetry \eqref{eq:fs} is motivated by applications to the study of asymptotic stability of solutions of nonlinear wave equations. More precisely, in \cite{kk11}, \cite{kk11b}, \cite{kk13}, %the operators $\cL$ has been studied in connection with the problem of asymptotic stability of solutions of various nonlinear equations. In particular, in \cite[\S 3]{kk13}, 
 the operator $\cL $ defined on $\gH=L^2(\R)\times L^2(\R)$ by \eqref{eq:cA} with 
\be\label{eq:kk}
A=-\frac{d^2}{dx^2}+m^2+V(x),\qquad C=\nu\frac{d}{dx},\qquad B=I,
\ee
was studied in connection with the problem of asymptotic stability of solutions of the nonlinear relativistic Ginzburg--Landau equation. The authors of \cite{kk13} were interested in the eigenfunction expansion properties for $\cL$, which was used in \cite{bs}, \cite{kk11b} for the calculation of the Fermi Golden Rule (this condition ensures a strong coupling of discrete and continuous spectral components of solutions, which provides the energy radiation to infinity and results in the asymptotic stability of solitary waves). If $|\nu|\in[0,1)$,  $V\to 0$ as $x\to \infty$ and under certain positivity assumptions on $A$ and $S_0$, it was shown in \cite{kk13} that the operator $\cL$ is positive and the eigenfunction expansion was constructed for all functions from the energy space $\gH_\cA$. However, the question on the eigenfunction expansion properties in the initial Hilbert space $\gH=L^2(\R)\times L^2(\R)$ was left to be open. It is one of our main aims to investigate this problem. 

On the other hand, under the assumptions of Hypothesis \ref{hyp:01}, the operator $\cL$ defined by \eqref{eq:cA} and \eqref{eq:kk} is definitizable (see Theorem \ref{lem:l_def}). Therefore (see \cite{lan84}), the problem on the eigenfunction expansion properties is equivalent to the regularity of critical points of the operator $\cL$. It turns out that the operator $\cL$  with coefficients \eqref{eq:kk} has a singular critical point $\infty$ (Theorem \ref{lem:l_def}). First of all, this result shows that the results obtained in \cite{kk13} are optimal in a certain sense. On the other hand, studying the spectral properties of the block operator matrix \eqref{eq:cA}, we are able to construct a class of $\cJ$-positive operators with the singular critical point $\infty$ (see Example \ref{ex:03}). The special case when all coefficients $A$, $B$ and $C$ are functions of a self-adjoint operator $T$ on $\cH$ (and hence they are commutative) was studied in \cite{j1}, \cite{j2}. %Note that in \cite{j1}, \cite{j2} the norm of spectral projections 

Let us now briefly describe the content of the paper. In Section \ref{sec:2}, we recall basic facts from \cite{shk95} and \cite{tre} on spectral properties of the operator $\cA$. Section \ref{sec:3} deals with the spectral properties of the $\cJ$-self-adjoint operator $\cL$. We describe the spectrum of $\cL$, provide sufficient conditions for its definitizability and obtain a necessary condition for the similarity of $\cL$ to a self-adjoint operator. We demonstrate our findings by examples. For instance, we present a class of $2\times 2$ block operator matrices with the singular critical point $\infty$. In the final Section \ref{sec:4}, we study the spectral properties of the operator $\cL$ defined by \eqref{eq:cA} and \eqref{eq:kk}. The main result of this section, Theorem \ref{lem:l_def} states that the operator $\cL$ is definitizable and $\infty$ is a singular critical point if the potential $V$ satisfies \eqref{eq:Vcomp}.   

\section{Self-adjointness of the operator matrix $\cA$}\label{sec:2}
   
In this section we collect some information on basic spectral properties of the operator $\cA$ defined by \eqref{eq:cA}. We begin with the following result from \cite{shk95} (see also \cite[Chapter II.2]{tre}).

\begin{theorem}[\cite{shk95}]\label{th:shkal}
Assume that the operators $A$, $B$, $C$ satisfy the assumptions of Hypothesis \ref{hyp:01}. Then the operator $\cA_0:\cH\times\cH \to \cH\times\cH$,
\be
\cA_0\begin{pmatrix}
f_1 \\
f_2
\end{pmatrix}:=\begin{pmatrix}
Af_1 + C^*f_2\\
Cf_1+Bf_2
\end{pmatrix},\quad f\in\dom(\cA_0):=\dom(A)\times\dom(S_0),
\ee
is essentially self-adjoint.
% Let the operator $\cA$ be given by \eqref
\end{theorem}

\begin{proof}
We shall give a proof because our further considerations rely on this construction. The proof is based on the Frobenius--Schur factorization
\be\label{eq:schur}
\cA_0-z=%\tilde{\cA}_0(z)\cG(z)=
\begin{pmatrix}
I & 0\\
C(A-z)^{-1} & I
\end{pmatrix}
\begin{pmatrix}
A-z & 0\\
0 & S(z)-z
\end{pmatrix}
\begin{pmatrix}
I & (A-z)^{-1}C^\ast\\
0 & I
\end{pmatrix},
\ee 
where
\be\label{eq:s(z)}
S(z)=B-C(A-z)^{-1}C^\ast,\quad z\in\rho(A);\qquad \dom(S(z))=\dom(S_0).
\ee
Assumption (ii) implies that the operators 
\[
F(z):=C(A-z)^{-1}\quad \text{ and}\quad G(z):=(A-z)^{-1}C^*
\]
 are bounded in $\cH$ whenever $z\in\rho(A)$. Moreover, $\dom(F)=\cH$ and the closure of $G$ is a bounded operator on $\cH$. Therefore, the operators 
\be
\cF(z)=\begin{pmatrix}
I & 0\\
C(A-z)^{-1} & I
\end{pmatrix},\qquad
\cG(z)=\begin{pmatrix}
I & (A-z)^{-1}C^\ast\\
0 & I
\end{pmatrix},
\ee
are bounded and boundedly invertible on $\cH\times\cH$. Noting also that $\cG(z)^\ast= \cF(z^\ast)$ and $\cG(z)\subset \cF(z^\ast)^\ast$, we conclude that the operator $\cA_0$ is essentially self-adjoint if and only if so is $S(0)=S_0$. It remains to exploit the assumption (iv).
\end{proof}
 
Using \eqref{eq:schur}, we can describe the closure of $\cA_0$.

\begin{corollary}[\cite{shk95}]\label{cor:a_clos}
Assume the conditions of Theorem \ref{th:shkal}. Then the closure $\cA$ of the operator $\cA_0$ is given by
\be\label{eq:A_clos}
\cA=\begin{pmatrix}
I & 0\\
F(0) & I
\end{pmatrix}
\begin{pmatrix}
A & 0\\
0 & \overline{S_0}
\end{pmatrix}
\begin{pmatrix}
I & F(0)^\ast\\
0 & I
\end{pmatrix},
\ee 
and 
\be\label{eq:doma}
\dom(\cA)=\{f=(f_1,f_2)^T:\ f_1+F(0)^\ast f_2\in\dom(A),\ f_2\in\dom(\overline{S_0})\}.
\ee 
\end{corollary}

We also need the following description of the spectrum of $\cA$. In what follows we shall use the following notation:
\begin{align}
\begin{split}
&\sigma(\overline{S}):=\{z\in\C:\, z\in\sigma(\overline{S(z)})\},\\
 &\sigma_i(\overline{S}):=\{z\in\C:\, z\in\sigma_i(\overline{S(z)})\},\quad i\in\{{\rm p,c,ess}\}.
\end{split}
\end{align}

\begin{corollary}\label{cor:A_spec}
Assume the conditions of Theorem \ref{th:shkal}. Let also $\cA$ and $\overline{S(z)}$ be the closures of $\cA_0$ and $S(z)$, respectively. Then 
\be\label{eq:s_A}
\sigma(\cA)\setminus\sigma(A)=\sigma(\overline{S}),\quad \sigma_i(\cA)\setminus\sigma(A)=\sigma_i(\overline{S}),\ \ i\in\{p,c\}.
\ee
In particular, the operator $\cA$ is (uniformly) positive if and only if so are $A$ and $S_0$. Moreover,
\be\label{eq:kappa}
\kappa_-(\cA)=\kappa_-(A)+\kappa_-(\overline{S_0}).
\ee
\end{corollary}

\begin{proof}
For the proof of equality \eqref{eq:s_A} we refer to \cite[Theorem 2.3.3]{tre}. The second claim is obvious since the operators $\cF$ and $\overline{\cG}$ are bounded, boundedly invertible and $\cF(z)^\ast=\overline{\cG(z^\ast)}$.
\end{proof}
 
\begin{remark}
Further assumptions on coefficients of $\cA$ are required in order to extend \eqref{eq:s_A} to the case of essential spectra. For instance, 
\[
\sigma_{\rm ess}(\cA)\setminus\sigma(A)=\sigma_{\rm ess}(\overline{S})
\]
if the operator $BA^{-1}$ is bounded on $\cH$. For further details and results we refer to \cite{shk95}, \cite[Chapter II.4]{tre}.
\end{remark}

\section{On the regularity of critical points of block operator-matrices}\label{sec:3}

Assume Hypothesis \ref{hyp:01}. Since $\cL_0=\cJ\cA_0$, the operator $\cL_0$ is essentially $\cJ$-self-adjoint if conditions (i)--(iv) of Hypothesis \ref{hyp:01} are satisfied. Moreover, its closure $\cL$ is given by $\cL=\cJ\cA$, where $\cA=\cA^*=\cA_0^\ast$. We can also  describe the closure using the Frobenius--Schur factorization. To this end, for all $z\in\C$ define the operator
\be
T(z)=B-(C+\I z)A^{-1}(C^*-\I z),\qquad \dom(T(z))=\dom(C^*)\cap\dom(B).
\ee 
Note that the assumptions (i)--(iv) imply that $T(z)$ is densely defined and $T(0)=S(0)$ is essentially self-adjoint. Straightforward calculations show that
\be\label{eq:schur_L}
\cL_0-z=
\begin{pmatrix}
I & - ( C + \I z )A^{-1}\\
0 & I
\end{pmatrix}
\begin{pmatrix}
0 & \I T(z)\\
-\I A &  0 
\end{pmatrix}
\begin{pmatrix}
I &  A^{-1}(  C^\ast - \I z)\\
0 & I
\end{pmatrix}
\ee
for all $f\in\dom(A)\times\dom(S_0)$. This representation enables us to find the closure of $\cL_0$ and also to describe its spectrum (see, e.g., \cite[Chapter II.3 and Theorem 2.4.16]{tre}).

\begin{theorem}\label{th:schur_L}
Assume Hypothesis \ref{hyp:01}. The closure $\cL$ of $\cL_0$ is given by
\be
\cL=
\begin{pmatrix}
I & -CA^{-1}\\
0 & I
\end{pmatrix}
\begin{pmatrix}
0 & \I  \overline{T(0)}\\
-\I A &  0 
\end{pmatrix}
\begin{pmatrix}
I & \overline{A^{-1}C^\ast}\\
0 & I
\end{pmatrix}
\ee
and $\dom(\cL)=\dom(\cA)$. Moreover,
\be
\sigma(\cL)=\sigma(\overline{T}),\qquad \sigma_i(\cL)=\sigma_i(\overline{T}),\ \ i\in\{{\rm p,c,ess}\},
\ee 
where 
\be
\sigma(\overline{T})=\{z\in \C:\, 0\in \sigma(\overline{T(z)})\},\qquad \sigma_i(\overline{T(z)})=\{z\in \C:\, 0\in \sigma_i(\overline{T})\}. %,\ \ i\in\{{\rm p,c,ess}\}.
\ee
\end{theorem} 

The next result is important for our further considerations.
%We left the proof of Theorem \ref{th:schur_L} to the reader. 
\begin{corollary}\label{cor:deftzbl}
Assume Hypothesis \ref{hyp:01}. Then the operator $\cL$ is definitizable if and only if there is $z\in \C$ such that 
$0\in \rho(\overline{T(z)})$. %the operator $\ove$
\end{corollary}

\begin{proof}
By \eqref{eq:kappa}, the form $\gt[f]:=\langle \cJ\cL f,f\rangle = \langle \cA f,f\rangle $, $f\in \dom(\cL)$, has finitely many negative squares. Therefore, by \cite[p.11, Example (c)]{lan84} (see also Corollary II.2.1 in \cite{lan84}), the operator $\cL$ is definitizable  if and only if $\rho(\cL)\neq \emptyset$. It remains to apply Theorem \ref{th:schur_L}.
\end{proof}

\begin{corollary}\label{cor:SpecL}
Assume Hypothesis \ref{hyp:01}. Then $\sigma(L)$ is symmetric with respect to the real line.

If additionally $\sigma(\cL)\neq \C$, then the non-real spectrum $\sigma(\cL)\setminus\R$ of $\cL$ consists of a finite number of pairs $\lambda$, $\lambda^*$. Moreover, total algebraic multiplicity of non-real eigenvalues is  at most $2\kappa_-(\cA)$. In particular, $\sigma(\cL)\subseteq \R$ if $\kappa_-(\cA)=0$, i.e., the operator $\cA$ is positive. % of finite 
\end{corollary}

\begin{proof}
The proof follows from \cite[Proposition II.2.1]{lan84} (see also \cite[Proposition 1.6]{cn}).
\end{proof}

Note that in practice the condition $\sigma(\overline{T})=\{z\in\C:\, 0\in\sigma(\overline{T(z)})\}\neq \C$ is difficult to check. Let us present two examples.

\begin{example}\label{ex:3.01}
Let $A$ be an unbounded self-adjoint uniformly positive operator in $\cH$, i.e., $A=A^*>0$ and $0\in\rho(A)$. Let also $C=0$ and  $B=A^{-1}$, that is,
\be
\cA=\begin{pmatrix}
A & 0\\ 0 & A^{-1}
\end{pmatrix},\qquad 
\cL=\begin{pmatrix}
0 & \I A^{-1}\\ -\I A & 0
\end{pmatrix}.
\ee
Clearly, all the assumptions of Hypothesis \ref{hyp:01} are satisfied. By Theorem \ref{th:schur_L}, 
\be
\sigma(\cL)=\sigma(\overline{T}),\qquad T(z) = A^{-1} - z^2 A^{-1} = (1-z^2)A^{-1}.
\ee
However, $T(\pm1)=0$ and $T(z)^{-1} = (1-z^2)^{-1} A$ for all $z\neq \pm 1$. Since $A$ is unbounded, $\sigma(\cL)=\C$ and hence the operator $\cL$ is not definitizable, however, it is $\cJ$-self-adjoint and $\cJ$-positive.

In particular, a very simple example of a $\cJ$-self-adjoint operator $\cL$ with $\sigma(\cL)=\C$ is given by
\be
\cL=\bigoplus_{n\in\N} \begin{pmatrix} 0 & \I/n \\ -\I n & 0 \end{pmatrix},\qquad \gH= l^2(\N;\C^2). 
\ee
\end{example}

\begin{example}\label{ex:3.02}
Let $a$, $b$ and $c:\R\to \C$ be locally integrable functions. Assume also that $a=a^*>0$, $b=b^*\ge 0$ a.e. on $\R$ and $1/a$, $c/a\in L^\infty(\R)$. Denote by $M_a$, $M_b$ and $M_c$ the multiplication operators in $L^2(\R)$ by $a$, $b$ and $c$, respectively, and set $A=M_a$, $B=M_b$ and $C=M_c$. Hence $\cA_0$ and $\cL_0$ are the operators on $L^2(\R)\times L^2(\R)$ defined by 
\be
\cA_0 = \begin{pmatrix} M_a & M_{c^*} \\ M_c & M_b \end{pmatrix}, \qquad  
\cL_0 = \begin{pmatrix} \I M_c & \I M_{b} \\ -\I M_a & -\I M_{c^*} \end{pmatrix}.
\ee 

Clearly, the operator $\cA=\overline{\cA_0}$ is self-adjoint and hence $\cL=\overline{\cL_0}=\cJ\cA$ is $\cJ$-self-adjoint. Let us also assume  that 
\be\label{eq:A_pos}
 a(x)b(x) - |c(x)|^2 \ge 0\quad\text{for a.a.  }\ \ x\in\R. 
\ee
The latter means that the operator $\cA$ is positive and $\cL$ is $\cJ$-positive. 

It is easy to see that under the assumptions on the coefficients $a$, $b$ and $c$, Hypothesis \ref{hyp:01} is satisfied. By Theorem \ref{th:schur_L}, the resolvent set of $\cL$ is given by
\be
z\in \rho(\cL)\quad  \Longleftrightarrow\quad \frac{a}{ab - (c + \I z)(c^* - \I z)} \in L^\infty(\R).
\ee
 
Moreover, in view of the positivity assumption \eqref{eq:A_pos}, the operator $\cL$ is definitizable and $\sigma(\cL)\subseteq \R$ if and only if $\I \in \rho(\cL)$, that is, 
\be\label{eq:3.11}
\frac{a}{ab - (c-1)(c^*+1)} \in L^\infty(\R).
\ee
\end{example}

Our main interest is the similarity of the operator $\cL$ to a self-adjoint operator. 
%From now on we shall assume that the operator $\cL$ is $\cJ$-positive, or, equivalently, $\kappa_-(\cA)=0$. We begin with some necessary conditions for the operator $\cL$ to be similar to a self-adjoint one.

\begin{lemma}\label{lem:lrg}
Assume Hypothesis \ref{hyp:01}. Let also $\sigma(\cL)\subseteq \R$. If the operator $\cL$ is similar to a self-adjoint operator, then there is a positive constant $K>0$ such that 
\be\label{eq:lrg_cL}
\|(\overline{T(z)})^{-1}\|_{\cH} \le \frac{K}{|\im z|},\quad \|A^{-1}(\overline{T(z)})^{-1}\|_{\cH} \le \frac{K}{|z||\im z|}
\ee
for all $z\in \C\setminus\R$.
\end{lemma}

\begin{proof}
Using the Frobenius--Schur factorization \eqref{eq:schur_L}, after straightforward calculations we find that the resolvent of $\cL$ is given by
\be
(\cL-z)^{-1} = \begin{pmatrix} 
\I \overline{A^{-1}(C^\ast - \I z)}(\overline{T(z)})^{-1} & -\I(\overline{A^{-1}(C^\ast-\I z)}(\overline{T(z)})^{-1} (C + \I z)A^{-1} + A^{-1})\\
-\I (\overline{T(z)})^{-1} & - \I (\overline{T(z)})^{-1}  (C + \I z)A^{-1}
\end{pmatrix}.
\ee
Note that, by Theorem \ref{th:schur_L}, $(\overline{T(z)})^{-1}$ is a bounded operator for each $z\in \C\setminus\R$ since $\sigma(\cL)\subseteq \R$. It remains to apply the resolvent growth condition (LRG), which states that 
\be\label{eq:lrg}
\|(\cL-z)^{-1}\|\le \frac{K}{|\im z|},\qquad z\in \C\setminus\R,
\ee
if $\cL$ is similar to a self-adjoint operator.
\end{proof}

 Lemma \ref{lem:lrg} enables us to construct a very simple example of a $J$-positive operator with the singular critical point infinity. 
 
 \begin{example}\label{ex:03}
 Let $A$ be a uniformly positive unbounded self-adjoint operator in $\cH$, $A=A^*\ge \varepsilon^2 I>0$. Let also $B=I$ and $C=0$, that is, the operator $\cL$ is given by % in \eqref{eq:cA}:
\be\label{eq:4.01}
\cL=\begin{pmatrix}
0 & \I I\\
-\I A & 0
\end{pmatrix},\qquad \dom(\cL)=\dom(A)\times\cH.
\ee 
Note that $\cL$ is $\cJ$-self-adjoint and $\cJ$-positive in $\gH=\cH\times\cH$. Moreover, 
\be
T(z) = I - z^2 A^{-1} = (A-z^2)A^{-1}, \qquad z\in\C.
\ee
Therefore, 
\be
\sigma(\cL)=\{\lambda\in\R:\ \lambda^2\in\sigma(A)\}\subseteq \R\setminus (-\varepsilon,\varepsilon).
\ee
Notice that $\infty$ is a critical point of $\cL$ since $A$ is unbounded. Moreover, we immediately find that
\[
\|T(z)^{-1}\| = \|A(A-z^2)^{-1} \| \ge 1
\]
for all $z\in \rho(\cL)$. By Lemma \ref{lem:lrg}, the operator $\cL$ is not similar to a self-adjoint operator (since it does not satisfy the LRG condition). Moreover, $\infty$ is a singular critical point of $\cL$. 
\end{example}

\begin{remark}
The results of Example \ref{ex:03} can be deduced from \cite{j1}, where the norms of spectral projections are computed in terms of coefficients of $\cL$ (see \cite[Satz 2.1.3]{j1}). %However, our proof is elementary. 
\end{remark}

\begin{remark}
Note that the operator \eqref{eq:4.01} provides a very simple example of a $\cJ$-positive operator with the singular critical point $\infty$. For instance, it suffices to take $\cH=L^2(\R_+;d\mu)$, where $d\mu$ is a positive Borel measure on $\R_+=(0,+\infty)$. Let also $A$ be the usual multiplication operator in $L^2(\R_+,d\mu)$
\[
(Af)(x)=(x+1)f(x),\quad x\in \R_+.
\]
If $\mu$ is a discrete measure, say $\mu=\sum_{n\in\N}\delta(x-n)$, then $L^2(\R_+)$ is equivalent to $l^2(\N)$ and the operator $A$ is simply the orthogonal sum of $2\times 2$ matrices
\[
\cL=\bigoplus_{n\in\N} \begin{pmatrix}
0 & \I\\
-\I n & 0
\end{pmatrix},\qquad \gH=l^2(\N;\C^2).
\]

Other simple examples of operators with singular critical points can be found in \cite[pp. 92--93]{cn}, \cite[Example 2.11]{gkmv}, 
\end{remark}

\begin{example}\label{ex:02con}
Let us continue with Example \ref{ex:3.02}. Assume additionally that the coefficients $a$, $b$ and $c$ satisfy \eqref{eq:3.11}. Then the operator $\cL$ is $\cJ$-positive and $\sigma(\cL)\subseteq \R$. Clearly, the resolvent of $\cL$ is given by 
\be
(\cL - z)^{-1} = \begin{pmatrix}
{\frac{-\I c^* - z}{(c+\I z)(c^* - \I z) -ab}} & {\frac{-\I b }{(c+\I z)(c^* - \I z) -ab}}\\
{\frac{\I a}{(c+\I z)(c^* - \I z) -ab}} & {\frac{\I c - z }{(c+\I z)(c^* - \I z) -ab}}
\end{pmatrix}.
\ee

If the operator $\cL$ is similar to a self-adjoint operator, then it satisfies the LRG condition \eqref{eq:lrg}. Clearly, the latter is equivalent to the following inequality 
\be\label{eq:3.20}
\Big\|\frac{|a| + |b| + |z|}{ab - (c-\I z)(c^*+\I z)}\Big\|_{L^\infty}\le \frac{K}{|\im z|}, \qquad z\in \C\setminus\R.
\ee
Here $K>0$ is a positive constant independent on $z$. 

For a detailed discussion of spectral properties of these operators we refer to \cite{j1} and \cite{j2}.% such that 
\end{example}

%The next result provides a similarity criterion for the multiplication operator considered in Examples \ref{ex:3.02} and \ref{ex:02con}.
%
%\begin{theorem}\label{th:3.10}
%Let $\cL$ be the operator considered in Example \ref{ex:3.02}. Let also the functions $a$, $b$ and $c$ satisfy all the assumptions of %Example \ref{ex:3.02} and condition \eqref{eq:3.11}. Then the operator $\cL$ is similar to a self-adjoint operator if and only if 
%\be
%\frac{a}{\sqrt{ab-|c|^2}},\ \ \frac{b}{\sqrt{ab-|c|^2}}\in L^\infty(\R).
%\ee
%\end{theorem}

 \section{Block matrices with differential operators}\label{sec:4}
 
 Let $V:\R\to \R$ be a locally integrable function, $V\in L^1_{\loc}(\R)$. The following operator arises in the study of stability of solitons for the 1-D relativistic Ginzburg--Landau equation (see \cite{kk11}, \cite{kk11b}, \cite{kk13}):
 \be
 \cL_0=\begin{pmatrix}
 \I\nu \frac{d}{dx} & \I I\\
 -\I(-\frac{d^2}{dx^2}+m^2+V(x)) & \I\nu\frac{d}{dx}
 \end{pmatrix},\quad \dom(\cL_0)=\cD(H_V)\times W^{1,2}(\R).
 \ee
Here $\cD(H_V)$ is the maximal domain of the operator $H_V=-\frac{d^2}{dx^2}+m^2+V(x)$
\be
\cD(H_V)=\{f\in L^2(\R):\, f,f'\in AC_{\loc}(\R),\, -f''+Vf\in L^2(\R) \}.
\ee
We shall assume (cf. \cite{kk11}, \cite{kk11b}, \cite{kk13}) that $\nu\in (-1,1)$, $m>0$ and %t is assumed that\marginpar{\tiny Assumptions are not optimal! }
\be\label{eq:Vcomp}
%V(x)=V(x)^\ast\ \ \text{a. e. on}\ \R,\qquad 
 \lim_{x\to \infty}\int_x^{x+1} |V(t)|dt =0. %\nu\in [0,1),\quad m>0,\quad V\in C^\infty(\R),\quad |V(x)|\le C\E^{-k|x|},
\ee
% Assume also for simplicity that $m^2+V(x)>0$ for a.a. $x\in\R$. 
Note that  condition \eqref{eq:Vcomp} implies that the potential $V$ is a relatively compact perturbation (in the sense of forms) of $H_0=-\frac{d^2}{dx^2}+m^2$ (cf. \cite[Chapter III.43]{gla}) and hence %Then 
\[
\sigma_{\rm c}(H_V)=\sigma_{\rm ess}(H_V)=[m^2,+\infty),\quad \kappa_-(H_V)=N<\infty.
%,\quad \sigma_p(H_V)=\{\lambda_1,\dots,\lambda_N\}\subset (0,m^2).
\]
Assume for simplicity that $z=0$ is not an eigenvalue of $H_V$. 
Then all conditions (i)--(iv) of Hypothesis \ref{hyp:01} are satisfied and hence we can apply the results from the previous sections. 

\begin{remark}\label{rem:nu=0}
If $\nu=0$, then the operator $\cL$ is a particular case of the operator considered in Example \ref{ex:03}. In this case the operator $\cL$ does satisfy the LRG condition \eqref{eq:lrg} and hence is not similar to a self-adjoint operator. We exclude this case from our further considerations. 
\end{remark}

Let $\psi_+(z,x)$ and $\psi_-(z,x)$ be the Weyl solutions of $-y''+(m^2+V(x))y=zy$ normalized such that $W(\psi_+,\psi_-)(z)=\psi_+(z,x)\psi_-'(z,x)-\psi_+'(z,x)\psi_-(z,x)=1$. Then the resolvent of the 1-D Schr\"odinger operator is given by
\be\label{eq:res_H}
(H_V-z)^{-1}f=\int_{\R}G(z;x,y)f(y)dy,\quad 
G(z;x,y)=\begin{cases} \psi_+(x) \psi_-(y), & y\leq x, \\ \psi_+(y) \psi_-(x), & y> x,    \end{cases}.
\ee

Denote $D=\frac{d}{dx}$, $\dom(D)=W^{1,2}(\R)$,  and assume that $0\in\rho(H_V)$. Then using \eqref{eq:res_H}, integration by parts shows that 
\[
DH_V^{-1}f=\int_\R G_x(0;x,y)f(y)dy,\quad \overline{H_V^{-1}D}f=-\int_\R G_y(0;x,y)f(y)dy,\ \ f\in L^2(\R),
\]
and
\[
\overline{S(0)}f=(1+\nu^2)f-\nu^2\int_{\R}G_{xy}(0;x,y)f(y)dy,\quad f\in L^2(\R).
\] 
Here the subscript denotes the partial derivative. Since $\overline{S(0)}$ is a bounded operator, the ranges of $DH_V^{-1}$ and $\overline{H_V^{-1}D}$ are contained in $W^{1,2}(\R)$.

Firstly, let us describe the spectral properties of the operator $\cA_0=\cJ\cL_0$ and its closure $\cA$.

\begin{lemma}\label{lem:4.03a}
Let $m>0$, $V$ satisfy \eqref{eq:Vcomp} and $0\in\rho(H_V)$. Then the operator $\cA_0$ is essentially self-adjoint and its closure is given by 
\be
\cA=
\begin{pmatrix}
I & 0\\
 \nu DH_V^{-1} & I
\end{pmatrix}
\begin{pmatrix}
H_V &  0 \\
0 & \overline{S(0)}
\end{pmatrix}
\begin{pmatrix}
I & -\nu \overline{H_V^{-1}D}\\
0 & I
\end{pmatrix}
\ee
on the domain 
\be
\dom(\cA)=\{f=(f_1,f_2)^T:\, f_1-\nu\overline{H_V^{-1}D}f_2\in \cD(H_V),\ f_2\in L^2(\R)\}.
\ee %$\dom(\cL)=W^{2,2}(\R)\times L^2(\R)$. 
%Here
%\[
%D=\frac{d}{dx},\qquad S(0)=I+\nu^2 DH_V^{-1}D.
%\]

The form domain of the operator $\cA$ is given by
\be\label{eq:5.07}
\dom(\cA^{1/2})=\{f=(f_1,f_2)^T:\, f_1\in W^{1,2}(\R),\ f_2\in L^2(\R)\}.
\ee
\end{lemma}

\begin{proof}
The first claim immediately follows from Corollary \ref{cor:a_clos}. To prove \eqref{eq:5.07} it suffices to note that 
\be
\dom(\cA^{1/2})=\{f=(f_1,f_2)^T:\, f_1-\nu \overline{H_V^{-1}D}f_2\in W^{1,2}(\R),\ f_2\in L^2(\R)\}.
\ee
However, $\overline{H_V^{-1}D}f_2\in W^{1,2}(\R)$ whenever $f_2\in L^2(\R)$.
 \end{proof}
 
The next result describes the essential spectrum of $\cA$ (cf. \cite[Lemma A.1]{kk13}).

\begin{corollary}
Assume the conditions of Lemma \ref{lem:4.03a}. Then 
\be\label{eq:a_ess}
\sigma_{\rm ess}(\cA)=\begin{cases}
%[m^2,+\infty)\cup\{1\}, & \nu=0,\\
[1-\nu^2,1]\cup[m^2,+\infty), & m\ge 1,\\
[1-\nu^2,m^2]\cup[1,+\infty), & 0\le 1-\nu^2\le m^2< 1,\\
[m^2,1-\nu^2]\cup[1,+\infty), & 0<m^2<1-\nu^2\le  1.
\end{cases}
\ee
\end{corollary}

\begin{proof}
It follows from  \eqref{eq:5.07} and \eqref{eq:Vcomp} that the operator $\mathcal{V}=V\oplus 0$ is a relatively compact perturbation (in the sense of forms) of the operator $\cA$ with $V\equiv 0$. %\marginpar{\tiny FIXME: explain..}%It is easy to see that the operator defined on $L^2(\R)\times L^2(\R)$ by 
%\[
%\begin{pmatrix}
%V & 0 \\
%0 & 0
%\end{pmatrix}\cA^{-1}
%\]
%is compact if $V$ satisfies \eqref{eq:Vcomp}. 
Therefore, by the version of Weyl's theorem for relatively compact perturbations, $\sigma_{\rm ess}$ does not depend on $V$ and hence we can set $V\equiv 0$.

To find the essential spectrum of the operator $\cA$ with $V\equiv 0$ let us apply the Fourier transform. Then the operator $\cA$ is equivalent to the multiplication operator $\hat{\cA}$ in $L^2(\R)\times L^2(\R)$ defined by $\hat{f}(\lambda)\to \hat{\cA}(\lambda)\hat{f}(\lambda)$, where
\be
\hat{\cA}(\lambda)=\begin{pmatrix}
1 & 0\\
 \frac{\I \nu\, \lambda}{\lambda^2+m^2} & 1
\end{pmatrix}
\begin{pmatrix}
\lambda^2+m^2 & 0\\
0 & 1- \frac{\nu^2\lambda^2}{\lambda^2+m^2}
\end{pmatrix}
\begin{pmatrix}
1 & \frac{ -\I \nu\,\lambda}{\lambda^2+m^2}\\
0 & 1
\end{pmatrix},\quad \lambda \in\R.
\ee 
Since the function $\hat{\cA}(\cdot)$ is continuous on $\R$, we conclude that $\sigma(\cA)=\sigma(\hat{\cA})=\sigma_{\rm ess}(\hat{\cA})$. Straightforward calculations show that 
\begin{align*}
(\hat{\cA}(\lambda) - z)^{-1}& = \frac{1}{\det(\hat{\cA}(\lambda) - z)} \begin{pmatrix}
1 & \I\nu\lambda\\
-\I\nu\lambda & \lambda^2+m^2
\end{pmatrix},\quad \lambda\in\R,\\
\det (\hat{\cA}(\lambda)-z)%=(\lambda^2+m^2-z)(1-z)-\nu^2\lambda^2
&=\lambda^2(1-z-\nu^2)+(m^2-z)(1-z).
\end{align*}
Therefore, $z\in\sigma(\hat{\cA})$ if and only if either $z=1-\nu^2$ or $\det (\hat{\cA}(\lambda)-z) = 0$ for some $\lambda\in\R$. 
Clearly, this equation has real solutions if and only if 
\[
\frac{(z-m^2)(z-1)}{z-(1-\nu^2)} \ge 0.
%(\lambda^2+m^2+1)^2-4(\lambda^2+m^2-\nu^2\lambda^2)<0.
\]
This completes the proof of \eqref{eq:a_ess}.
%It remains to note that ${\rm rank}(\hat{\cA}(\lambda)-m^2)=1$ for all $\lambda\in\R$ if $m^2=1-\nu^2$.
%This inequality completes the proof.
\end{proof}

\begin{corollary}\label{cor:4.04}
Assume that $0\notin\sigma(H_V)$. Then 
\be\label{eq:4.13}
\kappa_-(\cA)=\kappa_-(H_V)+\kappa_-(\overline{S(0)})<\infty.
\ee
In particular, $\cA$ is positive if and only if so are $H_V$ and $\overline{S(0)}$.
\end{corollary}

\begin{proof}
The first equality in \eqref{eq:4.13} follows from Corollary \ref{cor:A_spec}. Moreover, due to \eqref{eq:Vcomp}, $\kappa_-(H_V)=N<\infty$. It remains to show that $\kappa_-(\overline{S(0)})<\infty$. Denote by $S_0(0)$ the operator $S(0)$ with $V\equiv 0$. Note that $\sigma(\overline{S_0(0)})=\sigma_{\rm ess}(\overline{S_0(0)})=[1-\nu^2,1]$ (immediately follows by applying the Fourier transform). Moreover, 
\[
S_0(0)-S(0)=\nu^2DH_0^{-1}VH_V^{-1}D.
\] 
Note that the closure of this operator is compact on $L^2(\R)$ since $V$ satisfies \eqref{eq:Vcomp}. Therefore, $\sigma_{\rm ess}(\overline{S(0)})=\sigma(\overline{S_0(0)})=[1-\nu^2,1]\subset (0,1]$ since $|\nu| \in (0,1)$. This  implies the desired inequality.
\end{proof}

\begin{corollary}
The operator $\cA$ is nonnegative if and only if so are the operators $H_V$ and 
\be
H_{\nu,V}:=-(1-\nu^2)\frac{d^2}{dx^2}+m^2+V(x),\quad \dom(H_{\nu,V})=\cD(H_{\nu,V}).
\ee
\end{corollary}

\begin{proof}
By the previous corollary, it remains to show that $\overline{S(0)}\ge 0$ if and only if so is $H_{\nu, V}$. Next, the operator is positive if and only if %so is its quadratic form,
\[
\gt_S[f]=(\overline{S(0)}f,f)_{L^2}=(f,f)_{L^2}-\nu^2(H_V^{-1}Df,Df)_{L^2}>0 %,\quad f\in W^{1,2}_c(\R).
\] 
for all $f\in C^\infty_c(\R)$ (since this linear subspace is dense in $L^2(\R)$).
%Since $W^{1,2}_c(\R)$ is dense in $L^2(\R)$, it suffices to consider $\gt_T$ on functions from the Sobolev space 
%$W^{1,2}_c(\R)$. 
Setting $g(x)=f'(x)$ and integrating by parts once again, we finally get
\[
\gt_S[f]=(H_0^{-1}g,g)_{L^2}-\nu^2(H_V^{-1}g,g)_{L^2}>0,\quad g\in C^\infty_c(\R).
\] 
Here $H_0=-\frac{d^2}{dx^2}$ is the free Hamiltonian on $L^2(\R)$. The latter is equivalent to the positivity of the operator $H_{\nu,V}$.
\end{proof}

Now let us describe the spectral properties of the operator $\cL=\cJ\cA$. We begin with the description of the closure of $\cL_0$.
%Using Theorem \ref{}

\begin{lemma}\label{lem:4.03}
Assume the conditions of Lemma \ref{lem:4.03a}. Then the operator $\cL_0$ is essentially $\cJ$-self-adjoint and its closure is given by 
\be
\cL=
\begin{pmatrix}
I & -\nu DH_V^{-1}\\
0 & I
\end{pmatrix}
\begin{pmatrix}
0 & \I \overline{S(0)}\\
-\I H_V &  0 
\end{pmatrix}
\begin{pmatrix}
I & \nu \overline{H_V^{-1}D}\\
0 & I
\end{pmatrix}
\ee
on the domain 
\be
\dom(\cL)=\{f=(f_1,f_2)^T:\, f_1+\nu\overline{H_V^{-1}D}f_2\in W^{2,2}(\R),\ f_2\in L^2(\R)\}.
\ee %$\dom(\cL)=W^{2,2}(\R)\times L^2(\R)$. 
%Here
%\[
%D=\frac{d}{dx},\qquad T(0)=I+\nu^2 DH_V^{-1}D.
%\]
\end{lemma}

\begin{proof}
Note that $T(0)=S(0)$ and $\dom(\cL)=\dom(\cA)$. The rest of the proof follows from Theorem \ref{th:schur_L} and Lemma \ref{lem:4.03a}.
\end{proof}

As an immediate corollary of Theorem \ref{th:schur_L} we obtain the following description of $\sigma(\cL)$.%\marginpar{\tiny Incopmlete...}

\begin{corollary}\label{cor:4.07}
Assume the conditions of Lemma \ref{lem:4.03a} and set 
\be\label{eq:T}
T(z)=I+ (\nu D-\I z)H_V^{-1}(\nu D+\I z).
\ee 
Then 
\be
z\in\sigma(\cL)\ \ \Leftrightarrow \ \ 0\in \sigma(\overline{T(z)})\qquad (z\in\sigma_i(\cL)\ \ \Leftrightarrow \ \ 0\in \sigma_i(\overline{T(z)}),\ \ i\in\{{\rm p, c, ess}\}).
\ee
%\be
%\sigma_{\rm ess}(\cL)=\sigma_{\rm ess}(\cL_{V=0}).
%\ee
\end{corollary}

%\begin{remark}
%If we do not impose a positivity assumption on $V+m^2$, then using corollary we can easily show that $0\notin\sigma_{ess}(\cL)$ since $0\notin \sigma_{ess}(T(0))$.
%\end{remark}

%Note that the second equality follows from Birman's theorem about 

%\begin{hypothesis}\label{hyp:02}
%Suppose that $V$ is a bounded potential such that the Hamiltonian $H=-\frac{d^2}{dx^2}+m^2+V(x)$ is positive and boundedly invertible.
%\end{hypothesis}

\begin{theorem}\label{lem:l_def}
Assume the conditions of Lemma \ref{lem:4.03a}. Then the operator $\cL$ is difinitizable and $\infty$ is its singular critical point.
\end{theorem}

\begin{proof}
By Lemma \ref{lem:4.03a} and Corollary \ref{cor:4.04}, the operator $\cL$ is $\cJ$-self-adjoint and the form $\langle \cJ\cL\,\cdot\,,\cdot\,\rangle = \langle \cA\,\cdot\,,\cdot\,\rangle$ has finitely many negative squares. Therefore, $\cL$ is definitizable if $\rho(\cL)\neq \emptyset$. By Corollary \ref{cor:4.07}, we need to show that there is $z\in \C$ such that the operator $T(z)$ is boundedly invertible. Set $z=\I y$ with $y>0$. The operators $\nu D + y$ and $\nu D - y$ are boundedly invertible in $L^2(\R)$. Therefore, we get 
\be\label{eq:4.17}
(\nu D+y)^{-1} \overline{T(\I y)} (\nu D- y)^{-1} = (\nu^2D^2 - y^2)^{-1} + H_V^{-1}.
\ee
Since $-\nu^2 D^2 + y^2\ge y^2I$, we get $\|(\nu^2D^2 - y^2)^{-1}\|\le 1/y^{-2}$. Therefore, the left-hand side in \eqref{eq:4.17} is a boundedly invertible operator for $y>0$ sufficiently large since $0\in \rho(H_V)$. It remains to note that 
\be
(\overline{T(\I y)})^{-1} =(\nu D+y)^{-1} [(\nu^2D^2 - y^2)^{-1} + H_V^{-1}]^{-1} (\nu D- y)^{-1}. 
\ee
Therefore, $\I y\in \rho(\overline{T})$ for all sufficiently large $y>0$.

By \cite[Theorem 4.1]{cur}, $\infty$ is a singular critical point of $\cL$ if and only if $\infty$ is a singular critical point of $\cL$ with $V\equiv 0$. That is, it suffices to show that $\infty$ is a singular critical point of the operator $\cL=\overline{\cL_0}$, where $\cL_0$ is given in $L^2(\R)\times L^2(\R)$  by
\be
\cL_0= \begin{pmatrix}
\I \nu D & \I I \\
-\I H_0 & \I \nu D
\end{pmatrix}=\begin{pmatrix}
\I\nu\frac{d}{dx} & \I I \\
\I \frac{d^2}{dx^2} - \I m^2 & \I\nu\frac{d}{dx}
\end{pmatrix}.
\ee
 Now using the Fourier transform we see  that $\cL_0$ is unitarily equivalent to the multiplication operator acting in $L^2(\R)\oplus L^2(\R)$ and defined by
\be\label{eq:A0}
(\hat{\cL}_0 f)(\lambda)=\begin{pmatrix}
\nu \lambda  & \I \mathbf{1}\\
-\I (\lambda^2+m^2) & \nu \lambda 
\end{pmatrix}f(\lambda)= \begin{pmatrix}
\nu \lambda f_1(\lambda)+\I f_2(\lambda)\\
-\I(\lambda^2+m^2)f_1(\lambda) +\nu \lambda f_2(\lambda)
\end{pmatrix}.
\ee
The operator $\hat{\cL}_0$ is a particular case of the operator considered in Example \ref{ex:3.02} with $a(\lambda)=\lambda^2+m^2$, $b(\lambda)=1$ and $c(\lambda) = \nu\lambda$. Clearly, setting $z=\I y$ in \eqref{eq:3.20}, we get
\[
\Big\|\frac{a}{{ab - (c+ y)(c^*-y)}} \Big\|_{L^\infty}= \Big\|\frac{\lambda^2+m^2}{(1-\nu^2)\lambda^2+m^2 +y^2 - 2\I y \nu\lambda}\Big\|_{ L^\infty}\ge \frac{1}{1-\nu^2}.
\]
Therefore, there is no $K>0$ such that \eqref{eq:3.20} holds true. Hence the LRG test \eqref{eq:lrg} for the operator $\cL$ fails and hence $\cL$ is not similar to a self-adjoint operator. It remains to note that $\cL$ is $\cJ$-positive with $0\in\rho(\cL)$ if $V\equiv 0$. Therefore, $\infty$ is the only critical point of $\cL$. Since $\cL$ is not similar to a self-adjoint operator, $\infty$ is a singular critical point of $\cL$.
\end{proof}
 
%\section{Examples}

 \noindent
{\bf Acknowledgements.} I am deeply grateful to Branko \'Curgus, Andreas Fleige, Alexander Komech and Elena Kopylova for hints with respect to the literature and fruitful discussions. %I am also indebted to Anton Lunyov for critical remarks, which have helped to improve the exposition.      

\end{document}